\documentclass[12pt,a4paper]{amsart}
\usepackage{amssymb,latexsym,amsmath,amscd,graphicx,url,color}
\def\tensor{\mathcal}

\let\oldlabel=\label
\def\prellabel{\marginparsep=1em\marginparwidth=44pt
    \def\label##1{\oldlabel{##1}\ifmmode\else\ifinner\else
         \marginpar{{\footnotesize\ \\ \tt
                    ##1}}\fi\fi}}

\theoremstyle{plain}

\newtheorem{thm}{Theorem}[section]
\newtheorem{prop}[thm]{Proposition}
\newtheorem{cor}[thm]{Corollary}
\newtheorem{lemma}[thm]{Lemma}

\theoremstyle{definition}
\newtheorem{defn}[thm]{Definition}
\newtheorem{conj}[thm]{Conjecture}
\newtheorem{rmk}[thm]{Remark}

\newtheorem{ex}[thm]{Example}

\newcommand{\NN}{{\mathbb N}}
\newcommand{\PP}{{\mathbb P}}

\DeclareMathOperator{\ini}{in}

\DeclareMathOperator{\reg}{reg}
\DeclareMathOperator{\Seg}{Seg}

\usepackage[color=red!50!white,textsize=small,textwidth=2.3cm]{todonotes}
\begin{document}

% \title{A characteristic free approach to secant varieties of  triple Segre products}
\title{A characteristic free approach to secant varieties of  Segre embedding of $\PP^1\times \PP^{a-1} \times \PP^{b-1}$}

\author{Aldo Conca}
\email{conca@dima.unige.it}
\address{Dipartimento di Matematica, Universit\`a di Genova, Italy} 
\author{Emanuela De Negri}
\email{denegri@dima.unige.it}
\address{Dipartimento di Matematica, Universit\`a di Genova, Italy} 
\author{\v{Z}eljka Stojanac}
\email{zeljka.stojanac@gmail.com}
\address{Institute for Theoretical Physics, University of Cologne, Germany}

\begin{abstract} The goal of this short note is to study the secant varieties of the Segre embedding   of  the product $\PP^1\times \PP^{a-1} \times \PP^{b-1}$ by means of the standard tools of combinatorial commutative algebra.  We reprove and extend to arbitrary characteristic the results of Landsberg and Weyman \cite{LW} regarding the defining ideal and the Cohen-Macaulay property of the secant varieties. Furthermore we compute  their  degrees and give a bound for  their Castelnuovo-Mumford regularities, which are sharp in many cases. 
\end{abstract}

\maketitle
%%%%%%%%%%%%%%%%%%%%%%%%%%%%%%%%%%%%%%%%%%%%%
\section{Introduction}

The topic of this paper is the secant varieties of  the triple  Segre product  of  $\PP^1\times \PP^{a-1} \times \PP^{b-1}$. 
For basic facts  about tensor decomposition and Cohen-Macaulayness of secant varieties of  Segre products  we refer the reader to Oeding's  paper \cite{Oe} or to the book of Landsberg \cite{L}. 
We will work  with $3$-tensors of size $(2,a,b)$ with  $2\leq a\leq b$. The variety of rank-$1$ tensors is denoted by 
$\Seg(2,a,b)$  and corresponds to   the image of  the Segre embedding  of $\PP^1\times \PP^{a-1}\times \PP^{b-1}$  in $\PP^{2ab-1}$. We will further denote by  $\sigma_t(2,a,b)$ the 
 $t$-secant variety of $\Seg(2,a,b)$,  that is,  the variety of rank-$t$ tensors. The  defining ideal of   $\Seg(2,a,b)$ in $R=K[x_{ijk} : (i,j,k)\in [2]\times [a]\times [b] ]$    will be denoted by $I(a,b)$.  The ideal defining the secant variety $\sigma_t(2,a,b)$ will be denoted by  $I(a,b)^{\{t\}}$. 

 We recall that  $\Seg(2,a,b)$   is a well understood toric variety whose defining  ideal is the  Hibi ideal  of  $[2]\times[a]\times [b]$. 
In other words, equipping $[2]\times[a]\times [b]$ with the natural   structure of distributive lattice  one has: 
$$I(a,b)=( x_\alpha x_\beta-x_{\alpha\wedge \beta}x_{\alpha\vee \beta} :  \alpha,\beta  \mbox{ are incomparable elements of } [2]\times[a]\times [b] ).$$

Our goal it to prove that the ideal  $I(a,b)^{\{t\}}$  is generated by  determinantal equations coming from ``unfoldings"  of  the associated tensors and  that it defines a Cohen-Macaulay ring.  These results are already known in characteristic $0$, see \cite{LW}, while our results are valid in arbitrary characteristic.   Furthermore our construction allows us to give a formula for the degree of the secant varieties $\sigma_t(2,a,b)$, and to give a bound for their Castelnuovo-Mumford regularity; this bound   is sharp if  $2t\leq b$. 
We also conjecture that $I(2t,2t)^{\{t\}}$ defines a Gorenstein ring for all $t\geq 1$ and check it for 
$t\leq 3$.

The {\it unfolding} (a.k.a. flattening) is a transformation that reorganizes a tensor into a matrix.  For a tensor  of size  $n_1 \times n_2 \cdots \times n_d$, the $k$-th unfolding is a matrix  of size $n_k \times (n_1 \cdots n_{k-1} n_{k+1} \cdots n_d)$. The rows are indexed by the $k$-th index, and the columns are indexed by the vectors of the remaining indices. The order on the row and column indices is not important. In the following we will order them  lexicographically.

\begin{ex}\label{Ex:unfolding}
Let $\tensor{X}$ be an order-$3$ tensor of size $2 \times 3 \times 3$. Then, the first, second, and third unfolding are of the form 
$$\begin{matrix}
\tensor{X}^{\{1\}}= 
\begin{bmatrix}
 x_{111} & x_{112} &  x_{113} & x_{121} & x_{122} & x_{123}  & x_{131} & x_{132} & x_{133}  \\
 x_{211} & x_{212} &  x_{213} & x_{221} & x_{222} & x_{223} & x_{231} & x_{232} & x_{233}  
\end{bmatrix} \\  \\
 \tensor{X}^{\{2\}}= 
\begin{bmatrix}
 x_{111} & x_{112} &  x_{113} & x_{211} & x_{212} &  x_{213}  \\
 x_{121} & x_{122} &  x_{123} & x_{221} & x_{222} &  x_{223}  \\
 x_{131} & x_{132} &  x_{133} & x_{231} & x_{232} &  x_{233}  \\

\end{bmatrix} \\  \\ 
\tensor{X}^{\{3\}}= 
\begin{bmatrix}
 x_{111} & x_{121}  & x_{131} & x_{211} & x_{221}   & x_{231}  \\
 x_{112} & x_{122}  & x_{132} & x_{212} & x_{222}   & x_{232}  \\
 x_{113} & x_{123}  & x_{133} & x_{213} & x_{223}   & x_{233}  \\
\end{bmatrix}.
\end{matrix}
$$

Set $X=[x_{ij}]$, $Y=[y_{ij}]$  with $x_{ij}= x_{1ij}$, $y_{ij}= x_{2ij}$ for all $i \in [a]$, $j \in [b]$. 
We  remark that for  $\tensor{X}$  of size $2 \times a \times b$   the second and the third unfolding of $\tensor{X}$ are  the  block matrices 
$$
\tensor{X}^{\{2\}} = \begin{bmatrix} X & Y \end{bmatrix} 
\quad \text{and} \quad
\tensor{X}^{\{3\}} = \begin{bmatrix} X^T & Y^T \end{bmatrix}, 
$$
They can be combined in the single arrangement: 
$$\begin{array}{ll}
X & Y \\
Y
\end{array}.
$$
  \end{ex}

The  $(t+1)$-minors of the various  unfolding matrices are contained in the ideal defining the $t$-secant variety of the Segre variety  but, in general, one needs extra generators.  On the other hand we will see that  for the tensor of size $(2,a,b)$ the $(t+1)$-minors coming from the unfolding matrices are enough to generate $I(a,b)^{\{t\}}$. 

In the following for a matrix $M$ with entries in a ring $R$ and $t\in \NN$ we will denote by $I_t(M)$ the ideal of $R$ generated by all the $t$-minors of $M$.

\section{Gr\"obner bases and  secant ideals}

In this section we quickly recall  the results of Sturmfels and Sullivant \cite{SS} and of Simis and Ulrich \cite{SU} that we will use.

%%%%%%%%%%%%%%%%%%%%%%%%%%%%
\begin{defn}
Let $J_1,J_2,\ldots,J_r$ be ideals in a polynomial ring $R=K[x_1,x_2,\ldots,x_n]$ over a field $K$. Their {\it join} $J_1 * J_2 * \cdots * J_r$ is the elimination ideal in $R$
$$
(J_1({\bf{y}}_1)+\ldots+J_r({\bf y}_r) + (y_{1i}+y_{2i}+\ldots + y_{ri}-x_i:1\leq i \leq n))\cap R,
$$
where ${\bf y}_i=(y_{1i},y_{2i},\ldots,y_{ni})$ are new indeterminates  for all $i \in [r]$ and $J_i({\bf y}_i)$ is the image of  $J_i$ in $K[{\bf x},{\bf y}] = K[x_1,x_2,\ldots,x_n, y_{11}, y_{12}, \ldots, y_{rn}]$ under the map ${\bf x} \mapsto {\bf y}_i$.

We denote by  $I^{\{r\}} $ the $r$-fold join of $I$ with itself,  that is, 
$$
I^{\{r\}} = I * I * \cdots * I.
$$
\end{defn}

If $K$ is algebraically closed and  $I$ defines an  irreducible   variety $\mathbb{X} \subset \mathbb{P}^N$, then $I^{\{r\}}$ is the defining ideal of the 
 $r$th secant variety $\mathbb{X}^{\{r\}} $ of $\mathbb{X}$, defined as:
$$
\mathbb{X}^{\{r\}} = \overline{\bigcup_{P_1,\ldots,P_r \in \mathbb{X}} \langle P_1,\ldots, P_r \rangle}
$$
where $\langle P_1,\ldots, P_r \rangle$ denotes the linear span of the points $P_1,\ldots,P_r$ and the bar denotes the Zariski closure.

Let $\Delta$ be a simplicial complex and denote by $I_{\Delta}$ the associated  Stanley-Reisner ideal. 
The ideal $I_\Delta^{\{r\}}$ turns out to be  square-free, that is, 
$$ I_{\Delta}^{\{r\}} = I_{\Delta^{\{r\}}}$$  where  the complex $\Delta^{\{r\}}$ is described as follows:  

\begin{prop}\cite[Remark 2.9 ]{SS}\label{prop:Delta-r} 
Every face of $\Delta^{\{r\}}$ is the union of $r$ faces of $\Delta$. Thus 
if  $ F_1,F_2,\ldots,F_v$  are the  facets of $\Delta$, then the facets of $\Delta^{\{r\}}$  are the maximal subsets of the form $F_{i_1} \cup F_{i_2} \cup \ldots \cup F_{i_r}$ with $1\le i_1<\cdots< i_r\le v$.
\end{prop}

Moreover the $r$-fold join of  an initial ideal contains the initial ideal of  $r$-fold join ideal. 

\begin{thm}\cite[Corollary 4.2]{SS}\label{thm:switchingr}   Let $I$ be an ideal and $\prec$ a term oder.    Then  one has: 

$$ \ini(I^{\{r\}}) \subseteq \ini(I)^{\{r\}}$$
for all $r$. \end{thm}

If the equality holds for every $r$ then the term order is said to be  {\it delightful} for the ideal $I$.

%%%%%%%%%%%%%%%%%
\section{The ideal of two minors and the associated simplicial complex }

Let $K$ be a field and let $\tensor{X}=(x_{ijk})$ be the  tensor of indeterminates of size $2 \times a \times b$, with $2\le a\le b $. 

Let $R=K[x_{111}, x_{112},\ldots, x_{2ab}]$  be the  polynomial ring with indeterminates the entries of $\tensor{X}$. From now on we denote by $\tau$ any diagonal term order, that is,  a term order such that the initial term of every minor of $\tensor{X}^{\{2\}}$ and of $\tensor{X}^{\{3\}} $ is its diagonal term. For example the lexicographic order induced by $$x_{111}>x_{112}>x_{113}>x_{121}>\cdots>x_{1ab}>x_{211}>\cdots>x_{2ab}$$ is diagonal.

One can easily check that in this case the defining ideal of the Segre embedding   of $\PP^1\times \PP^{a-1}\times \PP^{b-1}$  in $\PP^{2ab-1}$, that is,   the Hibi ideal associated to the poset $[2]\times [a]\times [b]$,  is indeed the ideal generated by the $2$-minors of  the second  and third   unfolding,  
that is,   
$$I(a,b)=I_2(\tensor{X}^{\{2\}})+I_2(\tensor{X}^{\{3\}}).$$

For example, the Hibi relation $x_{213}x_{124}-x_{113}x_{224}$ is not a minor in $\tensor{X}^{\{2\}}$ or $\tensor{X}^{\{3\}}$  but can be written as a sum of a $2$-minor of  $\tensor{X}^{\{2\}}$  and a $2$-minor of  $\tensor{X}^{\{3\}}$:
$$x_{213}x_{124}-x_{113}x_{224}=(x_{213}x_{124}-x_{114}x_{223})+(x_{114}x_{223}- x_{113}x_{224}).$$ 
 
 We start by identifying a Gr\"obner basis for the ideal $I(a,b)$ itself.  
 
\begin{prop}\label{Prop:GroebnerI2}
The $2$-minors of the second and the third unfoldings of $\tensor{X}$ are a Gr\"obner basis of $I(a,b)$ with respect to any diagonal term order. 
\end{prop}
\begin{proof} It is not restrictive to consider the lexicographic order induced by $x_{111}>x_{112}>x_{113}>x_{121}>\cdots>x_{1ab}>x_{211}>\cdots>x_{2ab}$. We first consider the following subsets of $2$-minors in $I_2$ :
$$\begin{array}{l}
M^{\{1\}}= \{ x_{1 \alpha_2 \alpha_3} x_{2 \beta_2 \beta_3} -  x_{1\beta_2 \alpha_3}  x_{2 \alpha_2 \beta_3}: \alpha_2 < \beta_2\} \\
M^{\{2\}}= \{ x_{1 \alpha_2 \alpha_3}  x_{2 \beta_2 \beta_3} -  x_{1\alpha_2 \beta_3}  x_{2 \beta_2 \alpha_3}: \alpha_3 < \beta_3\} \\
M^{\{3\}}= \{ x_{1 \alpha_2 \alpha_3}  x_{1 \beta_2 \beta_3} -  x_{1 \beta_2 \alpha_3}  x_{1 \alpha_2 \beta_3}: \alpha_2 < \beta_2, \alpha_3 < \beta_3\} \\
M^{\{4\}}= \{ x_{2 \alpha_2 \alpha_3}  x_{2 \beta_2 \beta_3} -  x_{2 \beta_2 \alpha_3}  x_{2 \alpha_2 \beta_3}: \alpha_2 < \beta_2, \alpha_3 < \beta_3\},
\end{array}
$$
and set $M=M^{\{1\}}\cup M^{\{2\}} \cup M^{\{3\}} \cup M^{\{4\}}$.
Notice that $G^{\{2\}}=M^{\{1\}} \cup M^{\{3\}} \cup M^{\{4\}}$ is the set of the $2$-minors of the second unfolding and that $G^{\{3\}}=M^{\{2\}} \cup M^{\{3\}} \cup M^{\{4\}}$ is the set of the $2$-minors of the third unfolding. 

It is well known and easy to check that $G^{\{2\}}$ and $G^{\{3\}}$ are  Gr\"obner basis of the ideals they generate with respect to $\tau$, where $\tau$ is the lexicographic order. Hence it is enough to consider the $S$-polynomials $S(f,g)$ of $f,g$ with $f \in M^{\{1\}}$ and $g \in M^{\{2\}}$ whose initial monomials  $\ini(S(f,g))$ are not relatively prime. There are two cases.

First let $f=  x_{1 \alpha_2 \alpha_3}   x_{2 \beta_2 \beta_3} -   x_{1\beta_2 \alpha_3}  x_{2 \alpha_2 \beta_3} \text{ and }
g =  x_{1 \alpha_2 \alpha_3}  x_{2 \hat{\beta}_2 \hat{\beta}_3} -  x_{1\alpha_2 \hat{\beta}_3}  x_{2 \hat{\beta}_2 \alpha_3}$,
with $ \alpha_2 < \beta_2 $ and $\alpha_3 < \hat{\beta}_3.$
Thus $ S(f,g) = -  x_{1\beta_2 \alpha_3}  x_{2 \alpha_2 \beta_3}  x_{2 \hat{\beta}_2 \hat{\beta}_3} +  x_{1 \alpha_2 \hat{\beta}_3}  x_{2 \hat{\beta}_2 \alpha_3}  x_{2 \beta_2 \beta_3},$
and since $\alpha_2< \beta_2$, one has  $\ini(S(f,g)) =  x_{1 \alpha_2 \hat{\beta}_3}  x_{2 \beta_2 \beta_3}  x_{2 \hat{\beta}_2 \alpha_3}$. 
Dividing $S(f,g)$ by
$h_0=  x_{1 \alpha_2 \hat{\beta}_3}  x_{2 \beta_2 \beta_3} -   x_{1 \beta_2 \hat{\beta}_3}  x_{2 \alpha_2 \beta_3} \in M^{\{1\}}$ one gets
$
S(f,g)  = x_{2\hat{\beta}_2 \alpha_3} \cdot h_0 -  x_{2 \alpha_2 \beta_3} ( x_{1\beta_2 \alpha_3}   x_{2 \hat{\beta}_2 \hat{\beta}_3} -  x_{1 \beta_2 \hat{\beta}_3}   x_{2 \hat{\beta}_2 \alpha_3})$
which reduces to zero modulo $M$ since 
$x_{1\beta_2 \alpha_3}   x_{2 \hat{\beta}_2 \hat{\beta}_3} -  x_{1 \beta_2 \hat{\beta}_3}   x_{2 \hat{\beta}_2 \alpha_3} \in M^{\{2\}} \, .$

The second case is $
f =  x_{1 \alpha_2 \alpha_3}   x_{2 \beta_2 \beta_3} -  x_{1\beta_2 \alpha_3}  x_{2 \alpha_2 \beta_3} \text{ and } g=  x_{1 \hat{\alpha}_2 \hat{\alpha}_3}  x_{2 \beta_2 \beta_3} -  x_{1 \hat{\alpha}_2 \beta_3}  x_{2 \beta_2 \hat{\alpha}_3},$ with $\alpha_2 < \beta_2 $ and   $ \hat{\alpha}_3 < \beta_3.$
Thus $S(f,g)  = -  x_{1 \hat{\alpha}_2 \hat{\alpha}_3}  x_{1\beta_2 \alpha_3}  x_{2 \alpha_2 \beta_3} +  x_{1 \alpha_2 \alpha_3}  x_{1 \hat{\alpha}_2 \beta_3}  x_{2 \beta_2 \hat{\alpha}_3}. 
$

If $\hat{\alpha}_2>\alpha_2$, then $\ini(S(f,g))= x_{1 \alpha_2 \alpha_3}  x_{1 \hat{\alpha}_2 \beta_3}  x_{2 \beta_2 \hat{\alpha}_3}$ and dividing by $h_1=x_{1 \alpha_2 \alpha_3}  x_{2 \beta_2 \hat{\alpha}_3} -  x_{1 \beta_2 \alpha_3}  x_{2 \alpha_2 \hat{\alpha}_3} 
\in M^{\{1\}}$  leads to
$$S(f,g)  =x_{1 \hat{\alpha}_2 \beta_3} \cdot h_1-  x_{1\beta_2 \alpha_3} (  x_{1 \hat{\alpha}_2 \hat{\alpha}_3}   x_{2 \alpha_2 \beta_3} -   x_{1 \hat{\alpha}_2 \beta_3}  x_{2 \alpha_2 \hat{\alpha}_3})$$
which reduces to zero modulo $M$ since $x_{1 \hat{\alpha}_2 \hat{\alpha}_3}   x_{2 \alpha_2 \beta_3} -   x_{1 \hat{\alpha}_2 \beta_3}  x_{2 \alpha_2 \hat{\alpha}_3}\in M^{\{2\}}$.

To conclude we have to consider three more situations, and arguing as before.

If $\hat{\alpha}_2 < \alpha_2$, then $\ini(S(f,g))= x_{1 \hat{\alpha}_2 \hat{\alpha}_3}  x_{1\beta_2 \alpha_3}  x_{2 \alpha_2 \beta_3}$ and we get 
$S(f_2,g_2)= x_{1 \hat{\alpha}_2 \beta_3} \cdot h_4-  x_{1\beta_2 \alpha_3} h_3 $
with
$ h_3 =  x_{1 \hat{\alpha}_2 \hat{\alpha}_3}  x_{2 \alpha_2 \beta_3} -  x_{1 \hat{\alpha}_2 \beta_3}  x_{2 \alpha_2 \hat{\alpha}_3} 
\in M^{\{2\}}$ and 
$h_4  =  -  x_{1\beta_2 \alpha_3}  x_{2 \alpha_2 \hat{\alpha}_3} +  x_{1 \alpha_2 \alpha_3}  x_{2 \beta_2 \hat{\alpha}_3} 
\in M^{\{1\}}\, .$

The case $\hat{\alpha}_2=\alpha_2$ and $\hat{\alpha}_3 < \alpha_3$  works as the previous one. 

If $\hat{\alpha}_2=\alpha_2$  and $\hat{\alpha}_3 =\alpha_3$, then $\ini(S(f,g))=  x_{1 \alpha_2 \alpha_3} x_{1 \alpha_2 \beta_3}   x_{2 \beta_2 \alpha_3}$ and one has
$
S(f,g) =   x_{1 \alpha_2 \alpha_3} (h_5- h_6), 
$
where
$ h_5  =  x_{1 \alpha_2 \beta_3}  x_{2 \beta_2 \alpha_3} -  x_{1 \beta_2 \beta_3}  x_{2 \alpha_2 \alpha_3} 
\in M^{\{1\}} $ and 
$h_6= x_{1\beta_2 \alpha_3}  x_{2 \alpha_2 \beta_3} -  x_{1 \beta_2 \beta_3}   x_{2 \alpha_2 \alpha_3}
 \in M^{\{2\}} \, .$

Finally, if $\alpha_2=\hat{\alpha_2}$ and $\alpha_3<\hat{\alpha}_3$, then $\ini(S(f,g))= x_{1 \alpha_2 \alpha_3}  x_{1 \alpha_2 \beta_3}  x_{2 \beta_2 \hat{\alpha}_3}$) and we have
$S(f,g)  =  x_{1 \alpha_2 \beta_3} \cdot h_7 -  x_{1 \beta_2 \alpha_3} \cdot h_8$
where
$ h_7 =  x_{1 \alpha_2 \alpha_3}  x_{2 \beta_2 \hat{\alpha}_3} -  x_{1 \beta_2 \alpha_3}  x_{2 \alpha_2 \hat{\alpha}_3} 
\in M^{\{1\}} $ and $h_8 =  x_{1 \alpha_2 \hat{\alpha}_3}  x_{2 \alpha_2 \beta_3} -  x_{1 \alpha_2 \beta_3}  x_{2 \alpha_2 \hat{\alpha}_3}
 \in M^{\{2\}}.$

Thus in these three situations  $S(f,g)$  reduces to zero modulo $M$; this finishes the proof.
\end{proof}

As we have seen in Example \ref{Ex:unfolding} instead of looking at the second and the third unfolding of the  tensor $\tensor{X}$  separately, we can combine them into a single  arrangement: 
$$
W=
\begin{matrix}
X & Y \\
Y & 
\end{matrix}
$$
where $X=(x_{1ij})$ and $Y=(x_{2ij})$ are both of dimension $a \times b$.

Thus $I(a,b)$ is generated by the $2$-minors in the arrangement and the initial ideal $\ini(I(a,b))$ is generated by the $2$-diagonals in the arrangement. We introduce a partial order in the set of variables so that the $2$-diagonals are exactly the pairs of comparable elements. To this end,  we identify the set of the variables with: 
$$P= \{(i,j) : 1\leq i \leq a, 1 \leq j \leq 2b\}$$
associating $x_{1ij}$ to  $(i,j)$ and  $x_{2ij}$ to $(i,b+j)$. 

In $P$ we introduce the following partial order: 

$(x,y)\preceq  (z,t)$ if and only if 
$\begin{cases}  
(x,y)=(z,t)  \mbox{ or} \\
x<z,\ y<t  \mbox{ or} \\ 
y\le b,\ t\ge b+1,\ y<t-b 
\end{cases}$

With this notation, the generators of $\ini(I(a,b))$ are exactly the pairs of distinct comparable elements in $P$. 
Now,  $\ini(I(a,b))$, being a square-free monomial ideal, corresponds to a simplicial complex that we denote by $\Delta_2$. Here we use the index  $2$ to recall the generators of the associated ideal are the $2$-diagonals, that is,   the pairs of comparable elements in $P$. 

Recall that an {\it antichain} in the partially ordered set $P$ is a set of elements no two of which are comparable to each other. Therefore the elements of $\Delta_2$ are the antichains of $P$. 
Moreover a {\it saturated} (or {\it maximal}) antichain of $P$ is an antichain that is, maximal with respect to inclusions.  Therefore the saturated  antichains of $P$ are the facets of $P$. 

A subset $\{(h_1,k_1), \cdots, (h_s, k_s) \}$ of elements of $P$ is said to be a {\it path} with starting point $(h_1,k_1)$ and ending point $(h_s, k_s)$  if   $(h_{t+1}, k_{t+1}) - (h_t, k_t) \in \{(1,0), (0,-1)\}$ for all $t=1,\ldots, s-1$. 
We will represent $P$ with the matrix orientation, that is,  with  $(1,1)$ in the top left corner and $(a,2b)$ in the bottom right corner. With this representation a path consists of a sequence of steps to the left  and steps  down.

\begin{thm}\label{thm:Delta2Facets}
The facets of $\Delta_2\subset P$ correspond to the paths  starting in $(1, b+h)$ and ending in $(a,h)$ for some $h$ with  $1\le h\le b$.
\end{thm}

\begin{proof}
Let $F$ be a  path in $P$ starting in $(1, b+h)$ and ending in $(a,h)$, with $1\le h\le b$, and passing through $(k,b+1)$ and $(k,b)$. 
We first prove that $F$ is a facet. 
It is clear that no pair of points in $F$ is comparable, thus $F\in\Delta_2$.  Since we know that  $\Delta_2$ is the simplicial complex associated to $\ini(I(a,b))$ and $R/I(a,b)$ has dimension $a+b$, it follows that the facets of $\Delta_2$ have at most cardinality $a+b$.  Hence $F$ is a facet of $\Delta_2$. 

To conclude we show that every facet in $\Delta_2$ is a subset of a path  $F$ as described above. Let $E \subset P$ be in  $\Delta_2$, that is, $E$  does not contain any pair of comparable elements.
 Note that by the first condition of $\preceq$ we have that $E$ is an antichain, so it is of the form
 $$E=\{(i_1, j_1), (i_2, j_2), \ldots, (i_{\alpha}, j_{\alpha}), (h_1, k_1), (h_2, k_2), \ldots ,(h_{\beta}, k_{\beta})\},$$ with 
$i_1 \geq \ldots \geq i_{\alpha} \geq h_1 \geq \ldots \geq h_{\beta} \mbox{ and }  j_1  \leq \ldots \leq j_{\alpha}\leq b<b+1\leq k_1 \leq \ldots \leq k_{\beta}$.
Moreover we can assume that $j_{t+1}>j_t$ whenever  $i_{t+1}=j_t$, and  $k_{t+1}>k_t$ whenever  $h_{t+1}=h_t$.
We prove that $E$ is contained in a path $F$ starting in $(1,b+j_1)$ and ending in $(a,j_1)$.

If $\beta=0$, that is, $E=\{(i_1, j_1), (i_2, j_2), \ldots, (i_{\alpha}, j_{\alpha})\},$ then we can saturate the antichain, and add the points $(i_1+1,j_1),(i_1+2, j_1),\dots, (a, j_1)$ and the points $(i_{\alpha}-1,j_{\alpha}), (i_{\alpha}-2,j_{\alpha}), \dots,  (1,j_{\alpha}), (1,j_{\alpha}+1),(1,j_{\alpha}+2),\dots, (1, b+j_1)$, so that we obtain the path $F$. 
Analogously one concludes if  $\alpha=0$, that is, if $E=\{(h_1, k_1), (h_2, k_2), \ldots ,(h_{\beta}, k_{\beta})\}.$

Suppose now that both $\alpha$ and $\beta$ are not zero. Then $k_{\beta}\le b+ j_1$, otherwise it would be $j_1<k_{\beta}-b$, thus $(i_1,j_1)\prec (h_{\beta}, k_{\beta})$, which contradicts $E\in\Delta_2$. So we can saturate and add the points $(i_1+1,j_1),(i_1+2, j_1),\dots, (a, j_1)$ and the points $(h_{\beta}-1, k_{\beta}), (h_{\beta}-2, k_{\beta}), (1, k_{\beta}), (1, k_{\beta}+1), \dots, (1, b+j_1)$, so obtaining also in this case the facet $F$ containing $E$. This finishes the proof.
\end{proof}

\begin{rmk}
Let $F$ be a path in $P$, starting in $(1, b+h)$ and ending in $(a,h)$, with $1\le h\le b$, and passing through $(k,b+1)$ and $(k,b)$.
It is clear that $F$  corresponds to a path $F'$ in 
$P'= \{(i,j) : 1\leq i \leq a, 1 \leq j \leq 2b \quad \text{or} \quad  a+1 \leq i \leq 2a, 1 \leq j \leq b\}$  starting in  $(1, b+h)$ and ending in $(a+k,1)$, with $k\le a $, and  passing through $(k,b+1)$, $(k,b)$, and $(a,h)$ (see Figure \ref{Fig:Path}).
Thus one  can also rephrase the theorem in term of  paths in $P'$.  In the following we will use both the descriptions of $\Delta_2$, as subset of $P$ and of $P'$.
\end{rmk}
\begin{figure}[h]
%\centering
\includegraphics[width=0.4\textwidth, trim={135 485 65 135}]{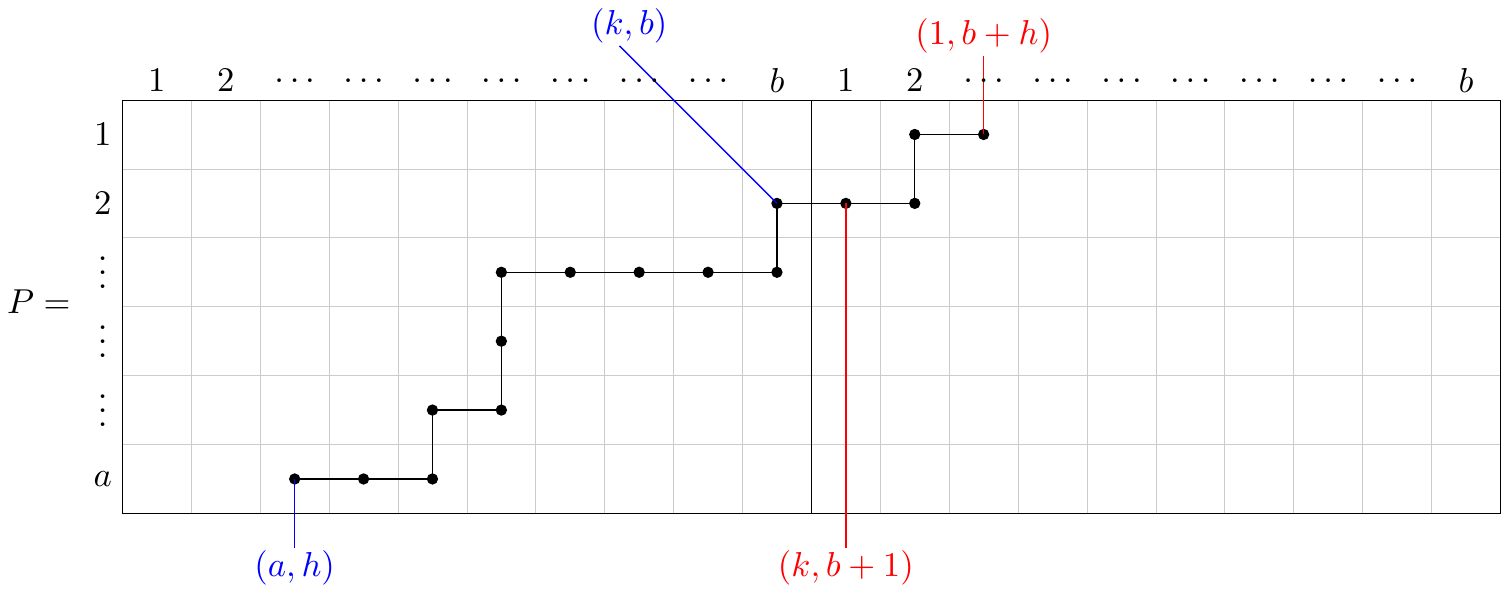}
%\centering 
\  \  \quad  \includegraphics[width=0.4\textwidth, trim={135 385 40 135}]{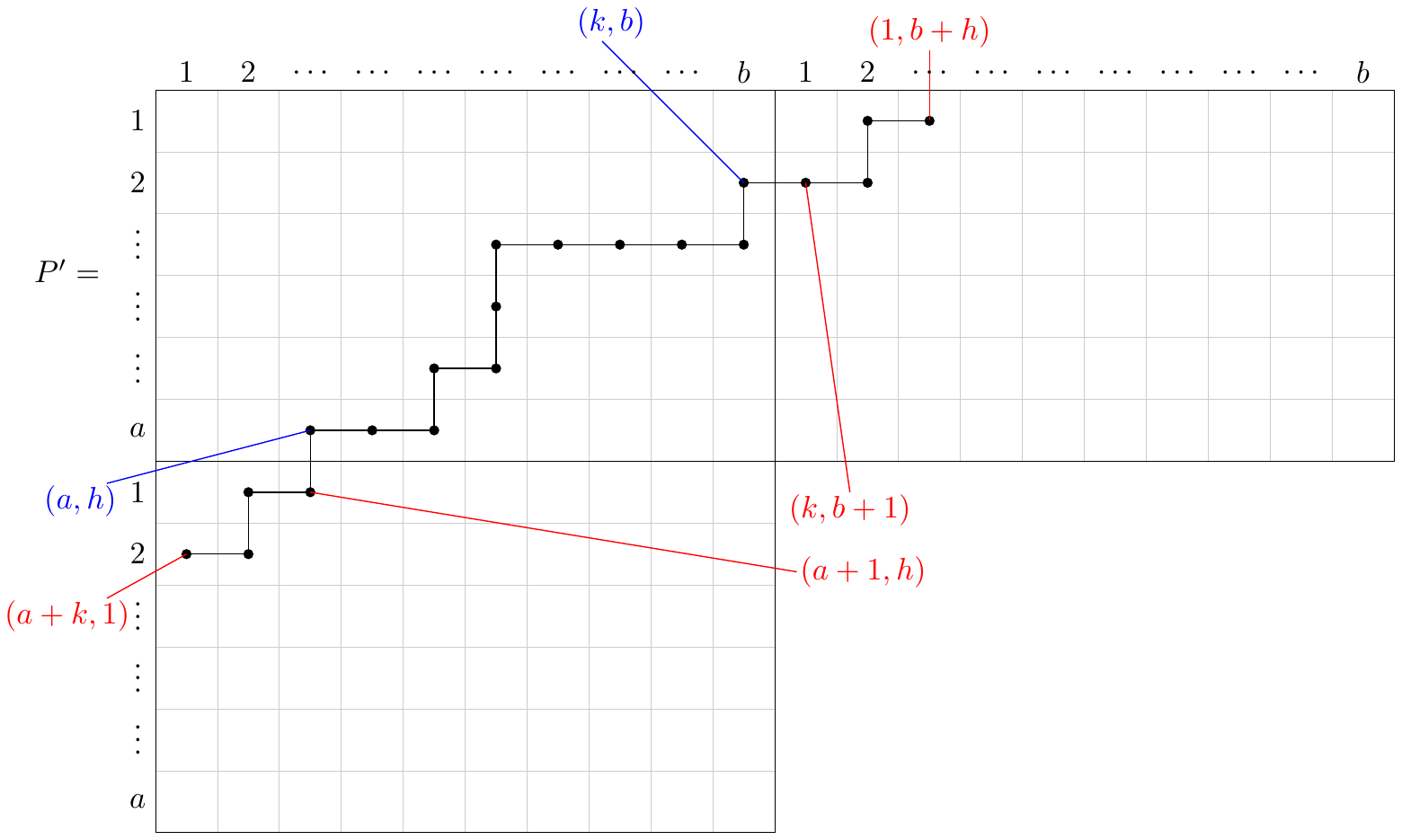}
\caption{
\label{Fig:Path} {\small An example of a path in $P$ and the corresponding path in $P'$. }
}
\end{figure}

\section{Secant ideals}

Our goal is to prove: 

\begin{thm} 
\label{Thm:main}
Let $K$ be a field of arbitrary characteristic and $a,b$ positive integers. Then 
\begin{itemize} 
\item[(1)] the defining ideal $I(a,b)^{\{t\}}$  of the $t$-secant variety $\sigma_t(2,a,b)$ of the Segre variety $\Seg(2,a,b)$ is    $$I_{t+1}(\tensor{X}^{\{2\}})+I_{t+1}(\tensor{X}^{\{3\}}).$$
\item[(2)] the $(t+1)$-minors of $\tensor{X}^{\{2\}}$ and $\tensor{X}^{\{3\}}$ are a Gr\"obner basis of  $I(a,b)^{\{t\}}$ with respect to any  diagonal term order. 
\item[(3)]  $R/I(a,b)^{\{t\}}$ is a Cohen-Macaulay domain. 
\end{itemize} 
\end{thm} 

To simplify the  notation in the following we will denote by $I_{t+1}$ the ideal $I_{t+1}(\tensor{X}^{\{2\}})+I_{t+1}(\tensor{X}^{\{3\}})$. 
  
The  ideal $I(a,b)$ defines $\Seg(2,a,b)$ and is generated by the $2$-minors of the arrangement  $W$, that is,  $I_2=I(a,b)$.  Hence the secant  variety $\sigma_t(2,a,b)$ is contained in the variety of   tensors whose second and third unfolding are of rank at most $t$.  Hence it follows that: 

\begin{lemma}\label{contained} The ideal $I_{t+1} $ is contained in $I_2^{\{t\}}$.
\end{lemma}

\begin{rmk}\label{rem:Mirsky}  Mirsky's theorem \cite{M} states that the size of the largest chain in a finite poset   equals the smallest number of antichains into which the  poset  may be partitioned. The simplical complex $\Delta_2^{\{t\}}$
consists of the subsets of $P$ that can be  decomposed into the union of $t$ antichains.  According to Mirsky's theorem, they are exactly the subsets of $P$ that do not contains chains of $t+1$ elements.  
Equivalently the ideal of $\Delta_2^{\{t\}}$ is generated by the chains of  $t+1$ elements, that is,  the leading terms of the $(t+1)$-minors of $W$. 
\end{rmk}

\begin{proof}[Proof of Theorem \ref{Thm:main} part (1) and (2)]
Denote by $J_{t} $ the ideal generated by the initial terms of the $(t+1)$-minors of the arrangement $W$, that is, the monomials correponding to the $t+1$-diagonals of $P'$. 
Notice that we have the following relations

\begin{equation} \label{eq:1}
J_t \subseteq \ini(I_{t+1})  \underset{\ref{contained}} \subseteq \ini(I_2^{\{t\}})  \underset{\ref{thm:switchingr}} \subseteq \ini(I_2)^{\{t\}}  
 = I_{\Delta_2}^{\{t\}} = I_{\Delta_2^{\{t\}}}  = J_t.   
\end{equation}

where the last equality follows from Remark~\ref{rem:Mirsky}. 
This proves that $J_t = \ini(I_{t+1})$ and $I(a,b)^{\{t\}}=I_{t+1}$. This concludes the proof of \ref{Thm:main}  part (1) and (2). 
\end{proof}

Note that all the inclusions in   Eq.(\ref{eq:1}) are indeed equalities. 
Thus one has:
\begin{cor}\label{delightful} 
The diagonal term orders are delightful for the ideal $I_2$.  
\end{cor}

To prove \ref{Thm:main}  part (3) we need to describe better the facets of the simplicial complex 
$\Delta_{t+1}$ associated to $\ini(I(a,b)^{\{t\}})$.  As a special case of  \ref{Thm:main}  part (2) we have that if  $t\geq \min(2a,b)$ the ideal $I(a,b)^{\{t\}}$ is trivial and if $a\leq t<\min(2a,b)$ then $I(a,b)^{\{t\}}$ is a generic determinantal ideal and hence well understood. So in the remaining part of the paper  we will assume that $t<a$. 

We recall this well-known result.
 
 \begin{thm} \cite[Theorem 1]{KS}\label{thm:KalkSturm}
Let $P$ be a prime ideal in a polynomial ring $R$,
 and let $\prec$ be any term order on $R$. Then the simplicial complex associated to $\ini_{\prec}(P)$ is pure of dimension $\dim(R/P)-1$ and strongly connected.
 \end{thm}

Denote by ${\mathcal F}(\Delta)$ the set of the facets of a simplicial complex $\Delta$.

\begin{lemma}\label{lem:Delta2r}
The ideal $I_{t+1}$ is prime. Moreover the simplicial complex  $\Delta_{t+1}$ is pure of dimension $\dim\Delta_{t+1}=(a+b)t-1$, with $t\leq a-1$, and 
\begin{equation}%\label{eq:facetsr}
{\mathcal F}(\Delta_{t+1}) = \{F_1 \cup F_2  \cup  \,\ldots\, \cup F_{t}:  F_i \in{\mathcal F}(\Delta_2) \mbox{ and } F_i\cap F_j=\emptyset  \mbox{ for } i\neq j\}.\end{equation}
 
\end{lemma}
\begin{proof}
The ideal $I(a,b)$ is prime  and so is its secant ideal $I_{t+1}$.  Thus  $\Delta_{t+1}$ is pure, since by Corollary \ref{delightful}  it is equal to   $\Delta_2^{\{t\}}$, which is pure by Theorem \ref{thm:KalkSturm}. Since any face of  $\Delta_{t+1}$  is the union of $t$ faces of $\Delta_2$  and the facets $\Delta_2$ have $a+b$ elements, to conclude it is enough to exhibit $t$ facets, say  $F_1,\dots, F_t$,  of $\Delta_2$ that are pairwise disjoint.  
 Let $F_1,\ldots, F_{t} \in \mathcal{F}(\Delta_2)$ be given as follows. 
Firstly $F_1$ is the ``right most" path in $P$ from $(1,2b)$ to $(a,b)$, that is, 
$$\{
(1,2b),  \ldots,   (a-1,2b),  (a,2b),  \ldots, (a,b+2)   (a,b+1),     (a,b)  \}.$$
Secondly   $F_2$ is the ``right most" path in $P$ from $(1,2b-1)$ to $(a,b-1)$ that does not intersect $F_1$ and   for $i=3,\dots, t$,  $F_i$ is the ``right most" path in $P$ from $(1,2b-i+1)$ to $(a,b-i+1)$ to  that does not intersect $F_{i-1}$.  For example  the picture \ref{Fig:rpaths} shows this construction in the extremal case  $t=a=6$, $b=10$. 
 
 \begin{figure}
 \includegraphics[width=0.6\textwidth, trim={135 495 65 135}]{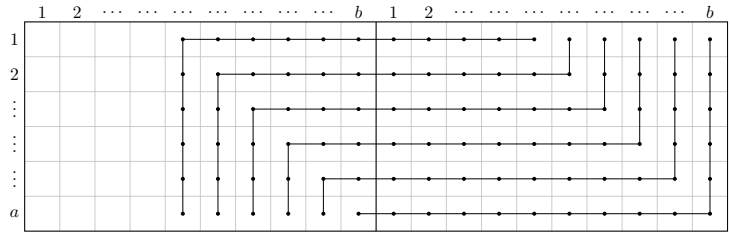}
\caption{
 \label{Fig:rpaths}
A facet  in $\Delta_7$ with $a=6$, $b=10$ }
\end{figure}
\end{proof}

In the following we say that an element  $(x,y)$ of a path $G$ in $P$ is a   low-right corner of   $G$  if  also $(x-1,y)$ and $(x,y-1)$ are in $G$. 
 By using the description of the facets of $\Delta_{t+1}$ given above  we prove the following. 

\begin{prop}\label{prop:shellability}
The simplicial complex $\Delta_{t+1}$  is shellable.
\end{prop}
\begin{proof} 
Given $x=(u,v)\in P$, we set 
$\mathcal{R}_x=\{(i,j)\in P \,:\, i>u, j>v\}$, 
$\bar {\mathcal{R}}_x=\{(i,j)\in P \,:\,  i\geq u, j\geq v\}$, 
$\mathcal{L}_x=\{(i,j)\in P \,:\,  i<u, j<v\}$, and
 $\bar{\mathcal{L}}_x=\{(i,j)\in P \,:\,  i\leq u, j\leq v\}$. Furthermore, if $G\subseteq  P$, we let $\mathcal{R}_G=\cup_{x\in G}\mathcal{R}_x$, $\bar{\mathcal{R}}_G=\cup_{x\in G}\bar{\mathcal{R}}_x$, $\mathcal{L}_G=\cup_{x\in G}\mathcal{L}_x$ and $\bar{\mathcal{L}}_G=\cup_{x\in G}\mathcal{L}_x$. Finally, we let $x^L=(u-1,v-1)$. 

First we define a partial order on the set of the facets of $\Delta_{t+1}$. Let $F=G_1\cup\ldots\cup G_{t}$ and $F'=G'_1\cup\ldots\cup G'_{t}$ be two facets of $\Delta_{t+1}$, where $G_i$ is a path starting in  $(1, b+h_i)$ and ending in $(a,h_i)$ and $G'_i$ is a path starting in  $(1, b+k_i)$ and ending in $(a,k_i)$. Assume that $h_1 < \cdots < h_{t} $ and  $k_1 < \cdots < k_{t} $. Set 
$$F \preceq F'  \hbox{\;\;\; if and only if \;\;\;} G'_i \subseteq \bar{\mathcal{R}}_{G_i} \hbox{\;\;\; for all \;\;\;} i=1,\ldots,t.$$
Extend $\preceq$ to a total order $\le$ on the set of the facets of $\Delta_{t+1}$. To
 prove that $\Delta_{t+1}$ is shellable, we need to show that, given any two facets $F'< F$, there exists
$x\in F\setminus F'$ and a  facet $F''$ such that $F''< F$ and $F\setminus F''=\{x\}$. 

Let $F' < F$ be two given facets decomposed as above. 
Since $F\not\preceq F'$,  we may consider the least 
integer $i$ such that $G'_i\not\subseteq\bar{\mathcal{R}}_{G_i}$, with $i\leq t$.
Thus there exists $y\in G'_i$ such that $y\not\in \bar{\mathcal{R}}_{G_i}$, that is, $y\in {\mathcal{L}}_{G_i}$. Choose such a $y$ with the smallest row index possible, and note that $\mathcal{R}_y\cap G_i\not=\emptyset$.

Now we consider two cases, according to the existence of a low-right corner $x$ of  $G_i$ which is in $\mathcal{R}_y$, cf. Figure \ref{Fig:shell}.

{\bf Case 1}: Assume that  such a corner does not exist. Let $y=(r,s)$; note that since $y\in \mathcal{L}_{G_i}$, then there exist $x'\in G_i$ such that $y\in \mathcal{L}_{x'}$, and by assumption  $x'=(r',k_i)$. In particular $k_i\le s<h_i$, thus $k_i<h_i$.
Note that $(1,b+h_i-1)$ is in $G_i$, because otherwise  we would have $(2,b+h_i)\in G_i$ and  thus  $y=(1, b+k_i)$, and there would be 
a low-right corner  of $G_i$ belonging to  $\mathcal{R}_y$, a contradiction.  Moreover  $(1,b+h_i)$ is not in $G_i'$. 

Now if  $i=t$ or if $i<t$ and $(1,b+h_i)\not\in G_{i+1}'$, then $(1,b+h_i)\in F\setminus F'$ and  $F''=(F\cup\{(a,h_i-1)\})\setminus\{(1,b+h_i)\}$ is facet such that $F''\preceq F$, thus $F'' < F$, as desired.  Otherwise, if $i<t$ and $(1,b+h_i)\in G_{i+1}'$, then $h_i\leq k_{i+1}$, and there exists an element  $y'\in G'_{i+1}$ such that $y'\not\in \bar{\mathcal{R}}_{G_{i+1}}$. So we can start again arguing on $y'$ as we did on $y$,  by using Case 1 or Case 2.

{\bf Case 2}: suppose that such a corner exists.
%Note that $x\not=(1, b+h_i)$,  since otherwise  it would be $y\in \mathcal{L}_{(1, b+h_i)}=\emptyset$, and that we can choose  $x\not=(a,h_i)$, because if $(a,h_i)$ would be the only low-right corner in $\mathcal{R}_y\cap G_i\not=\emptyset$, again we would have $k_i<h_i$, and get a contradiction.  
Since $y\in G'_i$ and the path of $F'$ are disjoint, one has that $y\in{\mathcal{R}}_{G'_{i-1}}$. By the minimality of $i$, $G'_{i-1}\subseteq \bar{\mathcal{R}}_{G_{i-1}}$, thus $y\in{\mathcal{R}}_{G_{i-1}}$.  Since $y\in \mathcal{L}_x$, it follows that $x^L\in\mathcal{R}_{G_{i-1}}$ and, consequently,
$x^L\not\in G_{i-1}$, thus  $x^L\not\in F$. 

Now if $x\not\in F'$, then  $F''=(F\cup\{x^L\})\setminus\{x\}$ is a facet of $\Delta_t$ such  that $F''\preceq F$, thus $F'' < F$, as desired. 

It remains to consider the case in which  $x\in F'$, that is,  $i<t$ and $x\in G_{i+1}'$. Since  $x \not\in \bar{\mathcal{R}}_{G_{i+1}}$, we can start again applying to $x$ the arguments we have used for $y$, following Case 1 or Case 2.

It is clear that this procedure concludes after a finite number of steps, proving that $\Delta_{t+1}$  is shellable.

%Now if $y=(1,s)$ for some $s$, then $G_i$ has no  low-right corner at all, so that $G_i$ has only one corner, which up-left, more precisely one has $G_i=\{(1, b+h_i),(1,b+h_i-1), ..., (1,h_i), (2,h_i),..., (a,h_i)$. Moreover  $k_i<h_i$. Thus the point  $x=(1,b+h_i)$ is in $F\setminus F'$. Now let $F''=(F\cup\{(a,h_i-1)\})\setminus\{x\}$. It is clear that $F''$ is a facet of $\Delta_t$ and that $F''\preceq F$, thus $F'' < F$, as desired. 

\end{proof}

\begin{figure}[h!]
 \includegraphics[width=7cm]{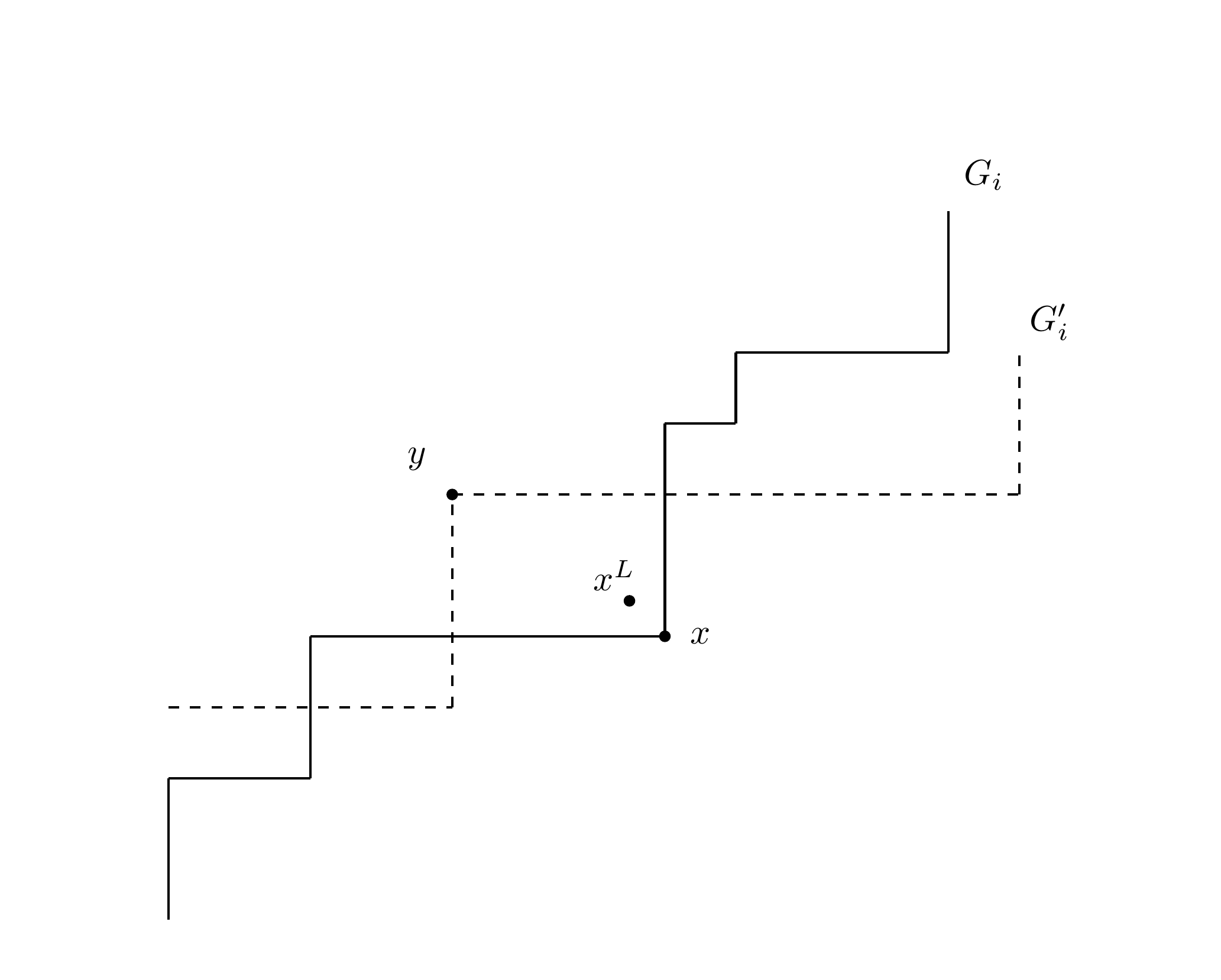}
\caption{
 \label{Fig:shell}
 }
\end{figure}

As a consequence we can conclude the proof of  Theorem \ref{Thm:main}. 

\begin{proof}[ Proof of \ref{Thm:main}  part (3)]  Since $\Delta_{t+1}$ is shellable the associate Stanley-Reisner ring $R/I_{\Delta_{t+1}}$ is Cohen-Macaulay \cite[Theorem 5.1.13]{BH}. Since $\ini(I(a,b)^{\{t\}})=I_{\Delta_{t+1}}$ it follows that $R/I(a,b)^{\{t\}}$ is Cohen-Macaulay as well. 
\end{proof} 
 
\begin{thm} 
\label{Thm:main2}
With the notation of \ref{Thm:main}, the ring  $R/I(a,b)^{\{t\}}$ is a domain of dimension $(a+b)t$ (with $t\leq a-1$), and of multiplicity
$$\sum_{1\leq h_1 < \cdots < h_{t} \leq b} \det (g_{ij}  )$$
with $g_{ij}= \binom{a+b-h_j+h_i-1}{a-1}$ for $1\leq i,j\leq t$.  Furthermore the Castelnuovo Mumford regularity of $R/I(a,b)^{\{t\}}$ is  less than or equal to $at$, with equality if $b\geq 2t$. 
\end{thm}

\begin{proof} Since $\Delta_{t+1}$ is pure, the multiplicity of $R/I(a,b)^{\{t\}}$  is equal to the number of the facets of $\Delta_{t+1}$. 
The set of the facets of  $\Delta_{t+1}$ can be written as the disjoint union of the set of non-intersecting paths  with starting points $Q_1=(1,b+h_1),\dots, Q_t=(1,b+h_t)$ and ending points $P_1=(a,h_1), \dots, P_t=(a,h_t)$ where $1\leq h_1 < \cdots < h_{t} \leq b$. 
 By a well known formula of  Gessel-Viennot \cite{GV}, the number of non-intersecting  paths  with starting points $Q_1,\dots, Q_t$ and ending points $P_1, \dots, P_t$  is the determinant of the matrix  whose $(i,j)$-th entry   is the number of  paths from $Q_i$ to $P_j$ which is easily seen to be equal to 
$$ \binom{a+b-h_j+h_i-1}{a-1}.$$
This gives the formula for the multiplicity. 
 
Given a shellable simplical complex $\Delta$ with a shelling $G_1,\dots, G_v$ we set 
$$r(G_i)= \{ x\in G_i : \mbox{ there exists } j<i \mbox{ such that } G_i\setminus G_j=\{x\} \}.$$
 If $K[\Delta]$ is the Stanley-Reisner ring associated to $\Delta$, then one has 
  $$\reg(K[\Delta])=\max\{ |r(G_i)| : i=1,\dots, v\}.$$
 Applying this to $\Delta_{t+1}$ and using the fact that $R/I(a,b)^{\{t\}}$ and $K[\Delta_{t+1}]$ have the same regularity, we have 
  $$\reg(R/I(a,b))= \max\{ |r(G)| :  G \mbox{ is a facet of } \Delta_{t+1} \}.$$
For a facet $G=\cup_{i=1}^t F_i$   and with respect to the shelling described in Proposition \ref{prop:shellability} one has that 
 $$r(G)\subseteq  \cup_{i=1}^t \{ x\in F_i : x \mbox{ is a right turn of $F_i$ or } x \mbox{ is the starting point of }F_i\}.$$
 Since any path $F_i$ has at most $a-1$ such turns, we have that $|r(G)|\leq at$ and hence $\reg(R/I(a,b)^{\{t\}})\leq at$. If $b\geq 2t$ then  one can actually find a facet $G$ of  $\Delta_{t+1}$ such that 
 $|r(G)|=at$ proving that for   $b\geq 2t$ one has $\reg(R/I(a,b)^{\{t\}})= at$.
 The construction  of  the facet $G$ such that $|r(G)|=at$ is illustrated by the following example, where $(a,b)=(5,8)$,  $t=4$, and a facet $G$ is represented with a dot in every  point belonging to $r(G)$: 

\noindent\mbox{\kern4.5em}\hfill
\includegraphics[width=0.6\textwidth, trim=135 565 135 125]{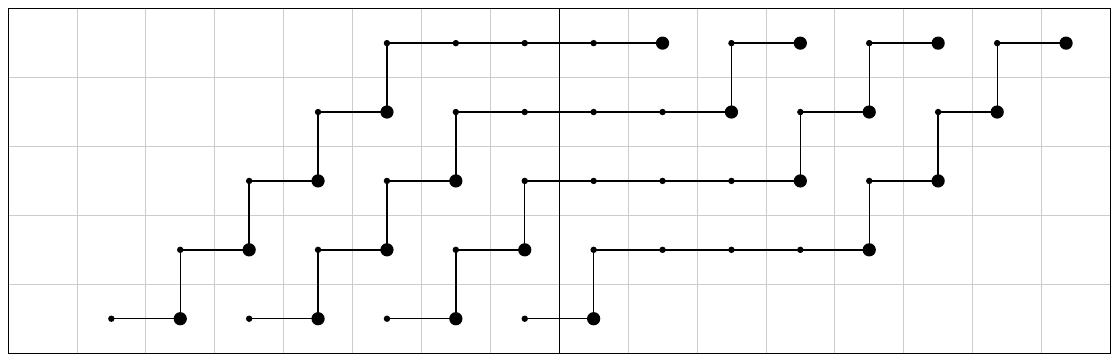}
\hfill\mbox{}\end{proof}
In the case $a=b=2t$ the facet with $ |r(G)|=at$ described in the proof is indeed the only facet of $\Delta_{t+1}$ with that property, as in this example, with $(a,b)=(6,6)$, $t=3$:    
 
\noindent\mbox{}\hfill
\includegraphics[width=0.6\textwidth, trim={165 545 165 125}]{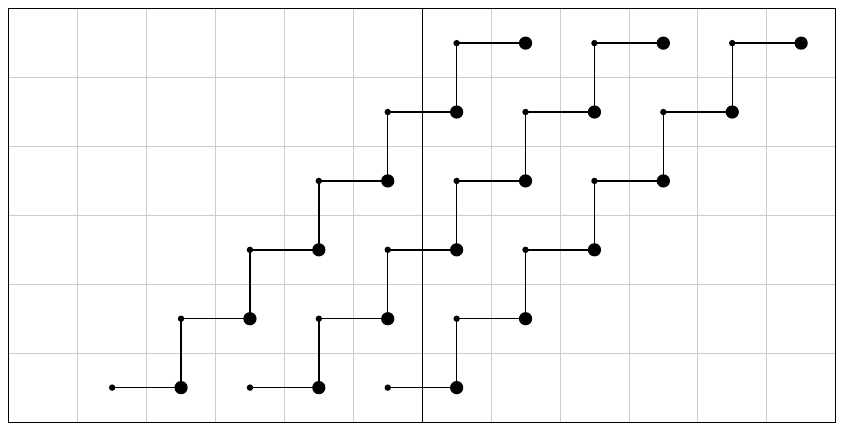}
\hfill\mbox{}

  This implies that the highest non-zero entry   of the $h$-vector of $\Delta_{t+1}$  is  $1$. The latter is a necessary condition for the $R/I(a,b)^{\{t\}}$ to be Gorenstein. This leads to the following conjecture: 
 
 \begin{conj}
 For any $t$ the ring  $R/I(2t,2t)^{\{t\}}$ is Gorenstein.
\end{conj}

For $t\leq 3$ one can actually compute the numerator of the Hilbert series of $R/I(2t,2t)^{\{t\}}$ and check that it is symmetric and hence verify, in view of \cite[Corollary 4.4.6]{BH}, that the conjecture holds.   For example for $t=3$ the numerator of the Hilbert series of $R/I(6,6)^{\{3\}}$ turns out to be: 
\begin{align*}
1 + 36x + 666x^2 + 8436x^3 + 68526x^4 + 366660x^5 + 1330644x^6 + 3296124x^7 + 5650866x^8\\   + 6762316x^9 + 5650866x^{10} + 3296124x^{11} + 1330644x^{12} + 366660x^{13} + 68526x^{14} + 8436x^{15} + \\666x^{16} + 36x^{17} + x^{18} 
\end{align*}

\section{Acknowledgements}
The work of \v{Z}eljka Stojanac  has been supported by the Excellence Initiative of the German Federal
and State Governments (Grant 81),  the DFG projects GRO 4334/1,2 (SPP1798 CoSIP) and  INdAM support during her one-month visit to the Department of Mathematics of the University of Genova in the spring of 2017.

%%%%%%%%%%%%%%%%%%%%%%%%%%%%%%%%%%%%%%%%%%%%%%%%%

\
 
 \bigskip
 
  \bigskip

 \

\end{document}